\documentclass[reqno, english]{amsart}
\usepackage{amsmath,amssymb,amsthm,bbm,mathtools,comment}
\usepackage[shortlabels]{enumitem}
\usepackage[pdftex,colorlinks,backref=page,citecolor=blue]{hyperref}
\usepackage[numbers]{natbib}
\usepackage{amsfonts}

\allowdisplaybreaks

\usepackage{cleveref}
\theoremstyle{plain}

\theoremstyle{plain}
\newtheorem{theorem}{Theorem}[section]		
\Crefname{theorem}{Theorem}{Theorems}

\newtheorem{lemma}[theorem]{Lemma}
\Crefname{lemma}{Lemma}{Lemmas}
\newtheorem{claim}[theorem]{Claim}
\Crefname{claim}{Claim}{Claimss}

\Crefname{proposition}{Proposition}{Propositions}

\Crefname{corollary}{Corollary}{Corollaries}
\newtheorem{conjecture}[theorem]{Conjecture}
\Crefname{conjecture}{Conjecture}{Conjectures}

\Crefname{problem}{Problem}{Problems}
\newtheorem{question}[theorem]{Question}
\Crefname{question}{Question}{Questions}
\theoremstyle{remark}
\Crefname{remark}{Remark}{Remarks}

\def\CC{\mathcal{C}}

\def\R{\mathbb{R}}

\def\PP{\mathcal{P}}

\let\emptyset\varnothing

\let\originalleft\left
\let\originalright\right
\renewcommand{\left}{\mathopen{}\mathclose\bgroup\originalleft}
\renewcommand{\right}{\aftergroup\egroup\originalright}

\makeatletter
\def\imod#1{\allowbreak\mkern10mu({\operator@font mod}\,\,#1)}
\makeatother

\usepackage{setspace}
\def \Np {N^+}
\def \Nm {N^-}
\def \Vp {V^+}
\def \Vm {V^-}
\def \ddp {d^+}
\def \ddm {d^-}
\def \dK {\overrightarrow{K}}

\usepackage{framed}

\newcommand{\ceil}[1]{
    \lceil #1 \rceil
}

\begin{document}
\title{Immersion of complete digraphs in Eulerian digraphs}
\author{Ant\'onio Gir\~ao}\thanks{Institut f\"ur Informatik, Universit\"at Heidelberg, Germany, E-mail: \texttt{a.girao}@\texttt{informatik.uni-heidelberg.de}.
\thanks{Supported by Deutsche Forschungsgemeinschaft (DFG, German Research Foundation) under Germany’s Excellence Strategy EXC-2181/1 - 390900948 (the Heidelberg STRUCTURES Cluster of Excellence)}}

\author{Shoham Letzter}
\thanks{
	Department of Mathematics, 
	University College London, 
	Gower Street, London WC1E~6BT, UK. 
	Email: \texttt{s.letzter}@\texttt{ucl.ac.uk}. 
	Research supported by the Royal Society.
}

\begin{abstract}

    A digraph $G$ \emph{immerses} a digraph $H$ if there is an injection $f : V(H) \to V(G)$ and a collection of pairwise edge-disjoint directed paths $P_{uv}$, for $uv \in E(H)$, such that $P_{uv}$ starts at $f(u)$ and ends at $f(v)$.
    We prove that every Eulerian digraph with minimum out-degree $t$ immerses a complete digraph on $\Omega(t)$ vertices, thus answering a question of DeVos, McDonald, Mohar, and Scheide. 
\end{abstract}

\maketitle

\section{Introduction}

    The study of the relation between the average degree of a graph and the existence of certain substructures, like minors or topological minors, has a long history. For example, in $1996$, Bollob\'as and Thomason \cite{bollobas1996highly} and, independently, Koml\'os and Szemer\'edi \cite{komlos1996topological} proved that any graph with average degree at least $ct^2$ (for a suitable constant $c$) must contain a subdivision of a clique on $t$ vertices. This is tight up to a constant factor; see K\"uhn and Osthus \cite{kuhn2006extremal} for a sharper bound on the required average degree (the best known lower bound is due to an observation of {\L}uczak).
    
    The analogous statement for digraphs is false. Indeed, there are digraphs with arbitrarily large minimum in- and out-degree which do not contain a subdivision of a complete directed graph on three vertices (see discussion below). Here, a \emph{complete digraph} on $k$ vertices, denoted $\dK_k$, is a digraph on $k$ vertices where between every two vertices there is an edge in both directions. Mader \cite{mader1996topological} asked whether large minimum out-degree guarantees the existence of a subdivision of a transitive tournament of given size, but it is still not known if this is true. 
    
    %In general, it is much harder to find certain structures in digraphs, for example it is still open a conjecture of Bollob\'as and Scott~\cite{bollobas1996dipaths} which states that an eulerian digraph with average degree $d$ contains a directed path of length $\Omega(d)$. 
    
    In this note, we consider a weakening of the concept of subdivisions, namely that of an \textit{immersion}. An \emph{immersion} of a (di)graph $H$ into a (di)graph $G$ is an injective mapping $f : V(H) \to V(G)$ and a collection of pairwise edge-disjoint (directed) paths $P_e$, one for each edge $e$ of $H$, such that the path corresponding to an edge $e = uv$ starts at $f(u)$ and ends at $f(v)$. 
    
    For undirected graphs, DeVos, Dvo{\v{r}}{\'a}k, Fox, McDonald, Mohar, and Scheide \cite{devos2014minimum} proved that average degree $200t$ guarantee an immersion of a clique on $t$ vertices. This was improved by Dvo\v{r}\'ak and Yepremyan \cite{dvovrak2018complete} to $11t + 7$ and by Liu, Wang, and Yang \cite{liu2020clique} to $(1 + o(1))t$ for $H$-free graphs, where $H$ is any fixed bipartite graph.
    
    Recently, Lochet \cite{lochet2019immersion} proved that a digraph with high enough minimum out-degree contains an immersion of a transitive tournament. Nevertheless, there are digraphs with arbitrarily large minimum in- and out-degree which do not contain an immersion of $\dK_3$ (see Mader \cite{mader1985degree}, who used a construction of Thomassen \cite{thomassen1985even} of a family of digraphs with arbitrarily large minimum out-degree with no even directed cycle; see also \cite{devos2012} for a different construction). 
    
    An \emph{Eulerian} digraph is a digraph where the in-degree of each vertex $u$ equals the out-degree of $u$.
    In \cite{devos2012,devos2013note}, DeVos, McDonald, Mohar, and Scheide showed that every Eulerian digraph with minimum out-degree at least $t^2$ contains an immersion of a $\dK_t$, and asked whether a linear lower bound on the minimum out-degree would suffice.
    Our main theorem confirms their belief.
    
    \begin{theorem} \label{thm:main}
        There is a constant $\alpha > 0$ such that for every integer $t \ge 1$, every Eulerian digraph with minimum in-degree at least $\alpha t$ contains an immersion of $\dK_t$.
    \end{theorem}
    
    \subsection{Overview of the proof}
        
        Our proof consists of three key steps. First, we use a notion of sublinear expansion introduced by Koml\'os and Szemer\'edi \cite{komlos1994topological,komlos1996topological}, which played a key role in recent progress on several long-standing conjectures (see, e.g.\ \cite{fernandez2022nested,haslegrave2021extremal,kim2017proof,liu2017proof,liu2020solution}). Our proof is somewhat unusual in that it applies this notion of expansion to digraphs. 
        To do so, we adapt the notion to our setting and prove that under appropriate assumptions, we can find an immersion of an Eulerian multi-digraph which is an expander with suitable properties; see \Cref{sec:expanders}.
        Next, we show that every such expander, with minimum in-degree $t$, immerses a simple digraph on $O(t)$ vertices with $\Omega(t^2)$ edges. Our proof of this is split into two lemmas, depending on the number of vertices in the expander; see \Cref{sec:find-dense-immersion}.
        Putting these two steps together, along with an additional observation, implies that every Eulerian digraph with minimum in-degree $t$ immerses an Eulerian multi-digraph on $O(t)$ vertices which has $\Omega(t^2)$ edges, ignoring multiplicities.
        The third and final step shows that such a digraph immerses a complete digraph on $\Omega(t)$ vertices. For this, we use the aforementioned result from \cite{devos2014minimum}, which shows that every graph with average degree $t$ immerses a complete graph on $\Omega(t)$ vertices; see \Cref{sec:immerse-clique}.
        We leave the proof of Theorem~\ref{thm:main} to Section~\ref{sec: proofmain}, and mention a few open problems in \Cref{sec:conclusion}.

\section{Preliminary lemmas} \label{sec:prelims}

    Recall that $\dK_k$ is the complete digraph on $k$ vertices. We define $\dK_{k,k}$ to be the digraph on $2k$ vertices that consists of two disjoint independent sets $A$ and $B$ of size $k$ and all edges from $A$ to $B$.
    We denote the in-degree of a vertex $u$ by $\ddm(u)$ and its out-degree by $\ddp(u)$. We will often consider multi-digraphs, in which case $\ddm(u)$ counts the number of in-neighbours of $u$ \emph{with multiplicities}. We denote the in-neighbourhood of $u$ by $\Nm(u)$ and the out-neighbourhood of $u$ by $\Np(u)$. Note that $\Nm(u)$ and $\Np(u)$ are sets, not multi-sets, so $|\Nm(u)|$ counts the number of in-neighbours of $u$ after \emph{ignoring multiplicities}.
    
    Given (multi-)(di)graphs $G$ and $H$, we say that $G$ \emph{immerses} $H$ if there is an injective mapping $f : V(H) \to V(G)$ and a collection of pairwise edge-disjoint paths $P_e$, for $e \in E(H)$, such that for an edge $e = uv$, the path $P_e$ starts at $u$ and ends at $v$. Equivalently, $G$ immerses $H$ if $H$ can be obtained from $G$ by performing a sequence of operations which take a (directed) path $xyz$ and replace its edges by the edge $xz$ (this operation is referred to as a \emph{split} and it is said that $y$ is \emph{split off}).
    
    We emphasise that when talking about a simple digraph, we mean a digraph in which there is at most one copy of $xy$ for every two vertices $x$ and $y$; in particular, a simple digraph may contain edges in both directions between a given pair of vertices.
    
    Logarithms are always taken modulo $2$. We drop rounding signs whenever they are not crucial.

    We recall the following result, due to DeVos, Dvo{\v{r}}{\'a}k, Fox, McDonald, Mohar, and Scheide \cite{devos2014minimum}, about immersions of complete graphs in graphs with large minimum degree.
    
    \begin{theorem}[\cite{devos2014minimum}] \label{thm:devos-et-al}
        Every simple graph with minimum degree at least $200t$ contains an immersion of $K_t$.
    \end{theorem}

    The following is a simple lemma that allows us to restrict our attention to Eulerian multi-digraphs, even after taking immersions.
    
    \begin{lemma} \label{lem:complete-to-eulerian}
        Let $D$ be an Eulerian digraph that immerses a digraph $D'$. Then, $D$ immerses an Eulerian multi-digraph $D''$ on the same vertex set as $D'$ that contains $D'$ as a subdigraph. 
    \end{lemma}
    
    \begin{proof}
        For each edge $e = xy$ in $D'$ let $P(e)$ be a directed path from $x$ to $y$ in $D$ such that the paths $P(e)$ with $e \in E(D')$ are pairwise edge-disjoint; such paths exist by definition of immersion. Let $D_0$ be the multi-digraph obtained from $D$ by replacing $P(e)$ by $e$ for each edge $e$ in $D'$; then $D_0$ is Eulerian. Write $D_0' = D'$. 
        
        We define multi-digraphs $D_i$ and $D_i'$, for $i \ge 1$, such that $D_i$ is Eulerian and $D_{i-1}'\subseteq D_i' \subseteq D_i$, as follows. 
        Suppose that $D_1, \ldots, D_i$ have been defined.
        If $D_i'$ is Eulerian, stop. Otherwise, let $x$ be a vertex in $D_i'$ whose out-degree is larger than its in-degree. Because $D_i$ is Eulerian and $D_i' \subseteq D_i$, there is a path $P$ in $D_i \setminus D_i'$ that ends at $x$ and starts at a vertex of $D'_i$ whose in-degree in $D_i'$ is larger than its out-degree. Let $P_i$ be a minimal subpath of $P$ that starts at a vertex with out-degree large than in-degree, and ends at a vertex with out-degree larger than in-degree in $D_i'$; denote its start and end vertices by $x_i$ and $y_i$. Form $D_{i+1}$ by replacing $P_i$ by the edge $x_i y_i$ and form $D_{i+1}'$ by adding $x_iy_i$ to $D_i'$. Then $D_i' \subseteq D_{i+1}' \subseteq D_{i+1}$ and $D_{i+1}$ is Eulerian. 
        
        Note that the sum $\sum_{u \in V(D_i)} |\ddp(u) - \ddm(u)|$ decreases as $i$ increases. Thus for some $i$ this sum will be $0$ and the process will stop. Then $D_i'$ is an Eulerian multi-digraph that contains $D'$ as a subdigraph and is contained in $D$ as an immersion.
    \end{proof}
    
    The next lemma will allow us to focus on \emph{regular} Eulerian digraphs.
    
    \begin{lemma} \label{lem:get-regular-Eulerian}
        Let $D$ be a simple Eulerian digraph with minimum in-degree at least $2d$. Then either $D$ immerses a simple $2d$-regular Eulerian digraph, or it immerses $\dK_{d,d}$. 
    \end{lemma}
    
    \begin{proof}
        Let $D'$ be a minimal (in terms of the number of edges) simple Eulerian digraph with minimum in-degree at least $2d$ which is immersed by $D$. If $D'$ is $2d$-regular, we are done, so suppose there is a vertex $u$ with in- and out-degree at least $2d+1$. Let $\Np \subseteq \Np(u)$ and $\Nm \subseteq \Nm(u)$ be disjoint sets of size $d$. Suppose that there exist $x \in \Np$ and $y \in \Nm$ such that $xy$ is not an edge in $D'$. Form $D''$ by removing the edges $xu$ and $uy$ from $D'$ and adding $xy$.
        Then $D''$ is Eulerian, it has minimum in-degree at least $2d$, and it is immersed by $D'$, implying that it is immersed by $D$. Since $D''$ has fewer edges than $D'$, this is a contradiction to the minimality of $D'$.
        It follows that $xy \in E(D')$ for every $x \in \Np$ and $y \in \Nm$. Hence, $D'$ contains a copy of $\dK_{d,d}$, implying that $D$ immerses $\dK_{d,d}$, as required.
    \end{proof}

\section{Expanders} \label{sec:expanders}
    
    In this section we introduce several notions of expanders. These are variants of the notions of `sparse expanders' introduced by Koml\'os and Szemer\'edi \cite{komlos1994topological,komlos1996topological} and `robust expanders' introduced by Haslegrave, Kim and Liu \cite{haslegrave2021extremal}.
    
    \subsection{Expanders in undirected graphs} \label{subsec:expanders-undirected}
    
        For $t > 0$ let $\rho_{t}$ be the function defined as follows (when $t$ is clear from the context, we often omit the subscript). 
        \begin{equation*}
            \rho_t(x) = 
            \left\{
                \begin{array}{ll}
                    0 & \text{if $x < t$}  \\
                    \frac{1}{256(\log (4x/t))^2} & \text{if $x \ge t$}. 
                \end{array}
            \right.
        \end{equation*}
        
        Denote the average degree of a graph $G$ by $d(G)$. 
        A graph $G$ is called a \emph{$t$-edge-expander} if every subset $X \subseteq V(G)$ with $|X| \le |G|/2$ satisfies $e_G(X, X^c) \ge 32d(G) \cdot \rho_t(|X|) |X|$. Similarly, $G$ is called a \emph{robust $t$-vertex-expander} if every subset $X \subseteq V(G)$ with $|X| \le |G|/2$ and subgraph $F \subseteq G$ with $e(F) \le d(G)\rho_t(|X|)|X|$ satisfy $|N_{G \setminus F}(X)| \ge 2\rho_t(|X|)|X|$.
        Haslegrave, Kim, and Liu \cite{haslegrave2021extremal} use a similar notion to the latter one (up to a different choice of constants); the former notion is more convenient for our application.
        
        The following lemma is a variant of similar lemmas such as Lemma 2.3 in \cite{komlos1996topological} and Lemma 3.2 in \cite{haslegrave2021extremal}. We prove it in \Cref{appendix} for completeness. Our proof is similar to the proofs of the aforementioned lemmas and also draws inspiration from the proof of Lemma 2.7 in \cite{jiang2021rainbow}. We note that we do not require the third item, but we keep it for future reference. 
        
        \begin{lemma} \label{lem:find-expander}
            Let $t > 0$ and let $G$ be a graph. Then there is a subgraph $H \subseteq G$ such that
            \begin{itemize}
                \item
                    $H$ has average degree at least $d(G)/2$ and minimum degree at least $d(G)/4$,
                \item
                    $H$ is a $t$-edge-expander,
                \item
                    $H$ is a robust $t$-vertex-expander.
            \end{itemize}
        \end{lemma}

    \subsection{Expanders in Eulerian digraphs} \label{subsec:expanders-eulerian}
    
        We now introduce analogous notions of expansion for digraphs (albeit with slightly different parameters).
    
        For a (multi-)digraph $D$, denote by $d(D)$ the average degree of $D$, counting multiplicities and ignoring directions. In other words, $d(D) = 
        e(D)/|D|$.
        Say that a multi-digraph $D$ is a \emph{directed $t$-edge-expander} if every subset $X \subseteq V(D)$ with $|X| \le |D|/2$ satisfies $e(X^c, X), e(X, X^c) \ge 4d(D)\rho_t(|X|)|X|$. Similarly, $D$ is a \emph{robust directed $t$-vertex-expander} if every subset $X \subseteq V(D)$ and subgraph $F$ with $e(F) \le d(D)\rho_t(|X|)|X|$ satisfy $|N^-_{D\setminus F}(X)|, |N^+_{D\setminus F}(X)| \ge \rho_t(|X|)|X|$.
        
        The following lemma is an analogue of \Cref{lem:find-expander} for simple directed Eulerian graphs.
        
        \begin{lemma} \label{lem:find-directed-expander}
            Let $t > 0$ and let $D$ be a simple $d$-regular Eulerian oriented graph (so every vertex has in-degree $d$). Then $D$ immerses an Eulerian multi-digraph $D'$ with the following properties.
            \begin{itemize}
                \item
                    The simple undirected graph obtained from $D'$ by ignoring directions and multiplicities has average degree at least $d/2$,
                \item
                    $D'$ has minimum in- and out-degree at least $d/8$ (taking into account multiplicities),
                \item
                    $D'$ is a directed $t$-edge-expander,
                \item
                    $D'$ is a robust directed $t$-vertex-expander.
            \end{itemize}
        \end{lemma}
        
        \begin{proof}
            Let $G$ be the undirected graph obtained from $D$ by ignoring directions; so $d(G) \ge d$ (if $xy$ and $yx$ are both in $D$, then we count $xy$ only once in $G$). By \Cref{lem:find-expander}, there is a subgraph $G' \subseteq G$ with average degree at least $d(G)/2$  and minimum degree at least $d(G)/4$, which is a $t$-edge-expander. Let $D'$ be a subgraph of $D$ which is an orientation of $G'$. Apply \Cref{lem:complete-to-eulerian} to find an Eulerian multi-digraph $D''$ which contains $D'$ as a subgraph and is contained in $D$ as an immersion. Observe that $D''$ has maximum in- and out-degree at most $d$, implying $d(D'') \le 2d \le 2d(G)$. We show that $D''$ satisfies the required conditions. 
            
            The first item follows from $G'$ having average degree at least $d(G)/2 \ge d/2$.
            
            Note that every vertex has either in- or out-degree at least $d(G)/8 \ge d/8$ in $D'$. Since $D''$ is an Eulerian digraph containing $D'$, the second item holds. 
            
            Let $X \subseteq V(D'')$. Then $e_{G'}(X, X^c) \ge 32d(G')\rho(|X|)|X|$, so one of $e_{D'}(X, X^c)$ and $e_{D'}(X^c, X)$ is at least $16d(G')\rho(|X|)|X|$. Since $D''$ is an Eulerian digraph that contains $D'$, we have $e_{D''}(X, X^c) = e_{D''}(X^c, X) \ge 16d(G')\rho(|X|)|X| \ge 8d(G)\rho(|X|)|X| \ge 4d(D'')\rho(|X|)|X|$. This shows that $D''$ is a $t$-edge-directed expander, as required for the third item.
            
            Let $F$ be a subgraph of $D''$ with $e(F) \le d(D'')\rho(|X|)|X|$ ($F$ can be a multi-digraph). Let $N = N^+_{D'' \setminus F}(X)$. Then $e_{D''}(X, N) \ge e_{D''} (X, X^c) - e(F) \ge 3d(D'')\rho(|X|)|X| > d \cdot \rho(|X|)|X|$, using $d(D'') \ge d(G') \ge d(G)/2 \ge d/2$.
            Since $D''$ has maximum in-degree at most $d$, it follows that 
            \begin{equation*}
                |N^+_{D'' \setminus F}(X)| 
                = |N| 
                \ge \frac{e_{D''}(X, N)}{d} 
                \ge \rho(|X|)|X|.
            \end{equation*}
            A symmetric argument shows $|N^-_{D'' \setminus F}(X)| \ge \rho(|X|)|X|$. This establishes vertex-expansion, as required for the fourth item.
        \end{proof}
    
    \subsection{Connecting sets in directed expanders} \label{subsec:expanders-connecting}
    
        The following lemma proves that robust directed vertex-expanders possess the following property, similarly to their undirected versions: for every two relatively large sets, there is a short directed path joining the two sets and avoiding a small set of `forbidden' edges. The proof is simple and similar to its undirected analogue (see, e.g.\ \cite{komlos1996topological}). We include the proof in \Cref{appendix} for completeness.
        
        \begin{lemma} \label{lem:short-paths-diam}
            Let $D$ be a multi-digraph which is a robust directed $t$-vertex-expander on $n$ vertices, where $n \ge 2^8t$. Let $X$ and $Y$ be two sets of size at least $x$, where $x \ge t$, and let $F$ be a subgraph of $D$ with at most $d(D) \rho(x)x$ edges. Then there is a directed path from $X$ to $Y$ avoiding $F$ of length at most $1600 (\log (n/t))^3$.
        \end{lemma}
      
\section{Immersions in Eulerian digraphs with high degree} \label{sec:find-dense-immersion}

    Our aim in this section is to prove \Cref{thm:find-dense-immersion}, which states that simple Eulerian digraphs with minimum degree at least $ck$, for a large constant $k$, immerse a dense simple digraph on at least $k$ vertices. The main works goes into showing that directed expanders with suitable properties immerse a dense subgraph of $\dK_{k,k}$. This is achieved in \Cref{lem:immersion-large-n}, which will be applied to relatively large expanders, and \Cref{lem:immersion-small-n}, which will be applied to smaller expanders.
    
    \subsection{Immersions in large directed expanders} \label{subsec:immersion-large-expanders}
    
        \begin{lemma} \label{lem:immersion-large-n}
            Let $k \ge 1$ be a sufficiently large integer and let $n > 4 k(100k)^{(\log \log n)^6}$. Let $D$ be a multi-digraph with the following properties: it is a robust directed $k$-vertex-expander on $n$ vertices; it has maximum in- and out-degree at most $100k$; and the graph obtained from $D$ by ignoring directions and multiplicities has at least $100kn$ edges. Then $D$ immerses $\dK_{k,k}$.
        \end{lemma}
        
        \begin{proof}
            Let $G$ be the simple graph obtained from $D$ by ignoring directions and multiplicities, and let $D'$ be a subgraph of $D$ which is an orientation of $G$; then $e(D') = e(G) \ge 100kn$. We claim that $D'$ has at least $n/2$ vertices with out-degree at least $3k$. Indeed, otherwise $e(D') \le 3k \cdot \frac{n}{2} + 100k \cdot \frac{n}{2} < 100kn$, a contradiction. Similarly, $D'$ has at least $n/2$ vertices with in-degree at least $n/2$. Let $\Vp$ be the set of vertices with out-degree at least $3k$ in $D'$ and let $\Vm$ be the set of vertices in $D'$ with in-degree at least $3k$. 
            
            Set $r = (\log \log n)^6$. We claim that there are disjoint sets of vertices $X \subseteq \Vp$ and $Y \subseteq \Vm$ of size $k$ each such that any two distinct vertices in $X \cup Y$ are at distance at least $2r+1$ from each other in $G$. To see this we define vertices $x_1, \ldots, x_k \in \Vp$ and $y_1, \ldots, y_k \in \Vm$, as follows. Let $x_1$ be any vertex in $\Vp$. Having defined $x_1, \ldots, x_{i-1}$, set $\Vp_i = \Vp \setminus (B_G(x_1, 2r) \cup \ldots \cup B_G(x_{i-1}, 2r))$. Since $|B_G(x_j, 2r)| \le (100k)^{2r}$ for $j \in [i-1]$, we have $|\Vp_i| > n - k(100k)^{2r} > \frac{n}{2}$. Let $x_i$ be any vertex in $V(G_i) \cap \Vp$. Define $y_1, \ldots, y_k$ similarly (we will need the inequality $n - 2k(100k)^{2r} > \frac{n}{2}$).
            
            Take $X = \{x_1, \ldots, x_k\}$ and $Y = \{y_1, \ldots, y_r\}$; then $X$ and $Y$ have the required property.
            For $x \in X \cup Y$, denote $B(x) = B_G(x, r)$; so the sets $B(x)$ with $x \in X \cup Y$ are pairwise disjoint.
            
            \begin{claim} \label{claim:expand-in-ball}
                Let $x \in X$ and let $F$ be a subgraph of $D$ with maximum in- and out-degree at most $\frac{k}{(\log \log n)^3}$. Let $F'$ be another subgraph of $D$, with the following property: write $D' = D \setminus F$ and let $X_i$ be the set of vertices $u$ for which there is a directed path in $D'$ of length at most $i$ from $x$ to $u$; then $F'$ contains at most $ki$ edges with both ends in $X_i$.
                Then $|X_r| \ge k(\log n)^6$. 
            \end{claim}
            
            \begin{proof}
                First note that $|X_1| \ge 3k - \frac{k}{(\log \log n)^3} - k \ge k$ (using $x \in \Vp$).
                
                We will prove the following.
                \begin{equation} \label{eqn:statement}
                    \text{If $i \in [2,r]$ and $|X_i| \le k(\log n)^6$, then $|X_{i+1}| \ge |X_i|(1 + \rho(|X_i|))$.}
                \end{equation}
                Assuming \eqref{eqn:statement} and $|X_r| \le k(\log n)^6$, then $|X_{i+1}| \ge |X_i|(1 + \rho(|X_i|))$ for $i \in [r]$. Using that $|X_1| \ge k$ and $\rho(|X_i|) \ge \rho(|X_r|) \ge \frac{1}{256(\log(4(\log n)^6))^2} \ge \frac{1}{(\log \log n)^3}$, we find that 
                \begin{equation*}
                    |X_r| 
                    \ge k\left(1 + \frac{1}{(\log \log n)^3}\right)^r 
                    \ge k\cdot \exp\left(\frac{r}{2(\log \log n)^3}\right)
                    > k \cdot \exp\left((\log \log n)^2\right)
                    > k(\log n)^6,
                \end{equation*}
                (recall that $r = (\log \log n)^6$) a contradiction. 
               
                We now turn to the proof of \eqref{eqn:statement}, which we prove by induction. Suppose that $i \in [2,r]$ and that $|X_{j+1}| \ge |X_j|(1 + \rho(|X_j|))$ for $j \in [2, i-1]$.
                
                Let $F_i$ be the subgraph of $F$ consisting of edges that are incident with $X_i$; similarly, let $F_i'$ be the subgraph of $F'$ consisting of edges incident with $X_i$. Note that 
                \begin{equation*}
                    e(F_i) 
                    \le \frac{k}{(\log \log n)^3} \cdot |X_i|
                    \le \frac{k \cdot \rho(k (\log n)^6)}{2} \cdot |X_i| 
                    \le \frac{k \cdot \rho(|X_i|)|X_i|}{2}.
                \end{equation*}
                We now prove that $e(F_i') \le \frac{1}{2} k \rho(|X_i|)|X_i|$. By choice of $F_i'$, we have $e(F_i') \le k(i+1) \le 2ki$. It thus suffices to show $4i \le \rho(|X_i|)|X_i|$. Write $s = |X_i|/k$. Then, using that $|X_1| \ge k$ and that $|X_{j+1}| \ge (1 + \rho(|X_j|))$ for $j \in [2, i-1]$,
                \begin{equation*}
                    s = \frac{|X_i|}{k} \ge \frac{|X_i|}{|X_1|}
                    \ge \left(1 + \rho(|X_i|)\right)^{i-1} 
                    \ge \left(1 + \frac{1}{256(\log (4s))^2}\right)^{i/2}
                    \ge \exp\left(\frac{i}{1024(\log (4s))^2}\right)
                \end{equation*}
                It follows that $4i \le 4096(\log (4s))^3$, implying that
                \begin{equation*}
                    \rho(|X_i|)|X_i|
                    = \frac{sk}{256(\log(4s))^2}
                    \ge 4096(\log(4s))^3
                    \ge 4i,
                \end{equation*}
                using that $k$ is large.
                Hence, indeed, $4i \le \rho(|X_i|)|X_i|$, as claimed. We thus have $e(F_i \cup F_i') \le k\cdot \rho(|X_i|)|X_i|$, so, by expansion, $|X_{i+1}| \ge |X_i|(1 + \rho(|X_i|))$, proving \eqref{eqn:statement}.  
            \end{proof}
            
            Write $a = k^2$ and let $(x_1, y_1), \ldots, (x_a, y_a)$ be an ordering of the ordered pairs $(x, y)$ with $x \in X$ and $y \in Y$. We pick paths $P_1, \ldots, P_a$ as follows. 
            
            Suppose that $P_1, \ldots, P_{i-1}$ are defined.
            We define subgraphs $F_{i, x_i}, F_{i, y_i}, F_{i, 0}, F_{i,1}, F_i \subseteq D$ as follows.
            The edges of $F_{i, x_i}$ are those that appeared in a path $P_j$ with $j < i$ that starts at $x_i$. Similarly, the edges of $F_{i, y_i}$ are those that appeared in a path $P_j$ with $j < i$ that ends at $y_i$. The edges of $F_{i, 0}$ are those appearing in a path $P_j$ with $j < i$. Form $F_{i,1}$ by including all edges in $D$ that are incident to a vertex $u \notin B(x_i) \cup B(y_i)$ which has in- or out-degree at least $\frac{k}{(\log \log n)^3}$ in $F_{i,0}$. Finally, set $F_i = F_{i, x_i} \cup F_{i, y_i} \cup F_{i, 0} \cup F_{i, 1}$. We take $P_i$ to be a shortest directed path from $x_i$ to $y_i$ in $D \setminus F_i$; we will show that such a path exists and has length at most $(\log n)^4$. 
            
            Suppose that $P_1, \ldots, P_{i-1}$ were chosen according the above procedure, and have length at most $(\log n)^4$.
            Write $x = x_i$, $y = y_i$, $F' = F_{i, x_i}$ and $F = F_i \setminus F'$. For $s \in [r]$ let $X_s$ be the set of vertices $u$ for which there is a path of length at most $s$ from $x$ to $u$ in $D \setminus F_i$. We claim that
            \begin{itemize}
                \item
                    the number of edges in $F'[X_s]$ is at most $ks$, for $s \in [r]$,
                \item
                    the maximum degree of $F[X_r]$ is at most $\frac{k}{(\log \log n)^3}$.
            \end{itemize} 
            To prove the first item, fix $s \in [r]$ and consider a pair $(x_j, y_j)$, with $j < i$, such that $x_j = x$. Note that $F_j[X_s] \subseteq F_i[X_s]$. Indeed, since $X_s \subseteq B(x)$ we have $F_j[X_s] = (F_{j, x} \cup F_{j, 0})[X_s] \subseteq (F_{i, x} \cup F_{i, 0})[X_s] = F_i[X_s]$. 
            Let $u$ be the last vertex of $P_j$ in $X_s$ and let $P'$ be the subpath of $P_j$ that starts at $x$ and ends at $u$. Then $P'$ is a shortest path in $(D \setminus F_j)[X_s]$ from $x$ to $u$ (otherwise $P_j$ could be replaced by a shorter path, contrary to its choice). As $F_j[X_s] \subseteq F_i[X_s]$ and thus $(D \setminus F_i)[X_s] \subseteq (D \setminus F_j)[X_s]$, it follows that $P'$ is a shortest path in $((D \setminus F_i) \cup P')[X_s]$, implying that $P_j$ contains at most $s$ edges with both ends in $X_s$. Since there are at most $k$ values of $j$ with $j < i$ and $x_j = x$, the first item above holds.
            
            We now prove the second item. Fix $u \in X_r$. Consider the largest $j$, with $j < i$, such that $u$ is in $P_j$ and $x_j \neq x$. By definition of $P_j$ and $F_{j,1}$ and the fact that $u \notin B(x_j) \cup B(y_j)$ this means that the in- and out-degrees of $u$ in $F_{j,0}$ are smaller than $\frac{k}{(\log \log n)^3}$. It follows that $u$ has in- and out-degree at most $\frac{k}{(\log \log n)^3}$ in $F[X_r]$, as required.
            
            Having proved the two items above, \Cref{claim:expand-in-ball} implies that $|X_r| \ge k(\log n)^6$. A symmetric argument implies that the set $Y_r$, of vertices $u$ for which there is a directed path in $D \setminus F_i$ from $u$ to $y$ of length at most $r$, has size at least $k(\log n)^6$.
             
            To complete the proof we need an upper bound on $e(F_i)$. Recall that the paths $P_1, \ldots, P_{i-1}$ have length at most $(\log n)^4$ each. Thus $e(F_{i, 0}) \le k^2(\log n)^4$. By choice of $F_{i,1}$ and the assumption that maximum in- and out-degree of $D$ is at most $100k$, it follows that
            \begin{equation*}
                e(F_i) 
                \,\le\, \frac{e(F_{i,0}) \cdot 200k}{k/(\log \log n)^3}
                \le k^2(\log n)^4 \cdot 200(\log \log n)^3
                \le k^2 (\log n)^5,
            \end{equation*}
            using that $n$ is large. Let $X_r'$ and $Y_r'$ be subsets of $X_r$ and $Y_r$, respectively, of size exactly $k(\log n)^6$.
            Then
            \begin{equation*}
                d(D) \cdot \rho(|X_r'|)|X_r'| 
                \ge k \cdot \frac{1}{256(\log(4(\log n)^6))^2} \cdot k(\log n)^6
                \ge k^2(\log n)^5 \ge e(F_i).
            \end{equation*}
            By \Cref{lem:short-paths-diam}, there is a path in $D \setminus F_i$ from $X_r'$ to $Y_r'$ of length at most $1600(\log n)^3$. It follows that there is a path from $x$ to $y$ in $D \setminus F_i$ whose length is at most $1600(\log n)^3 + 2r \le (\log n)^4$. This means that $P_i$ can be chosen appropriately and has length at most $(\log n)^4$, as claimed, for $i \in [a]$. 
            
            The paths $P_1, \ldots, P_a$ are pairwise edge-disjoint and the join each of the pairs $(x, y)$ with $x \in X$ and $y \in Y$. In particular, the union $P_1 \cup \ldots \cup P_a$ is an immersion of $\dK_{k,k}$.
        \end{proof}
        
    \subsection{Immersions in small expanders} \label{subsec:immersion-small-expanders}
    
        \begin{lemma} \label{lem:immersion-small-n}
            Let $n, k \ge 1$ be integers such that $k$ and $n/k$ are sufficiently large and $k \ge (\log (n/k))^7$. Let $D$ be a multi-digraph with the following properties:
            it is a robust directed $k$-vertex-expander on $n$ vertices; it has maximum in- and out-degree at most $100k$; and the graph obtained from $D$ by ignoring directions and multiplicities has at least $100kn$ edges. Then $D$ immerses a subgraph of $\dK_{2k,2k}$ with at least $k^2/2$ edges.
        \end{lemma}
        
        \begin{proof}
            \def \Up {U^+}
            \def \Um {U^-}
            \def \Us {U^{\sigma}}
            \def \Bp {B^+}
            \def \Bm {B^-}
            \def \Bs {B^{\sigma}}
            \def \up {u^+}
            \def \um {u^-}
            \def \us {u^{\sigma}}
            \def \Ns {N^{\sigma}}
            \def \Wp {W^+}
            \def \Wm {W^-}
            \def \Ws {W^{\sigma}}
            \def \pm {\{+, -\}}
            Write $\ell = n/k$; so $\ell$ is large and $k \ge (\log \ell)^7$. 
            Write $a = \min\{k, \frac{\ell}{(\log \ell)^8}\}$ and $b = \ceil{k/a}$; so $k \le ab \le 2k$.
            
            \begin{claim}
                There are sets of vertices $\Up_1, \ldots, \Up_b, \Um_1, \ldots, \Um_b$, and $W(u)$ for $u \in \bigcup_i(\Up_i \cup \Um_i)$, with the following properties.
                \begin{itemize}
                    \item 
                        The sets $\Up_1, \ldots, \Up_b, \Um_1, \ldots, \Um_b$ are pairwise disjoint sets of size $a$ each. Set $\Us := \Us_1 \cup \ldots \cup \Us_b$ for $\sigma \in \{+, -\}$ and $U := \Up \cup \Um$.
                    \item
                        The sets $W(u)$, with $u \in \Us_i$, are pairwise disjoint, for $i \in [b]$ and $\sigma \in \{+, -\}$. 
                    \item
                        $W(u)$ is a subset of $\Ns(u)$ of size $20k$, for $u \in \Us$.
                \end{itemize}
            \end{claim}
            
            \begin{proof}
                For some $i \in [b]$, suppose that $\Up_1, \ldots, \Up_{i-1}, \Um_1, \ldots, \Um_{i-1}$ and $W(u)$, for $u \in U_{< i}$, satisfy the above properties, where $U_{< i} = \bigcup_{j < i}(\Up_i \cup \Um_i)$. We will show how to obtain sets $\Up_i, \Um_i$ and $W(u)$ for $u \in \Up_i \cup \Um_i$, that together with the previously defined sets satisfy the above properties.
                
                Let $D' = D \setminus U_{< i}$. We will pick distinct vertices $u_1, \ldots, u_a$ in $D'$ and sets of vertices $W_1, \ldots, W_a$ that are pairwise disjoint sets of size $20k$, such that $W_j$ is a set of out-neighbours of $u_j$ in $D'$, for $j \in [a]$.
                Suppose that $u_1, \ldots, u_{j-1}$ and $W_1, \ldots, W_{j-1}$ are defined and satisfy the requirements, for some $j \in [a]$. We will show that a vertex $u_j$ and set $W_j$ with the required properties can be found. To see this, set $W_{< j} = \bigcup_{s < j} W_s$ and $D'' = D' \setminus (\{u_1, \ldots, u_{j-1}\} \cup W_{< j})$. 
                Note that $|U_{< i} \cup \{u_1, \ldots, u_{j-1}\}| \le 2ab \le 4k \le n/8$ and $|W_{< j}| \le a \cdot 20k \le \frac{20n}{(\log \ell)^8} \le n/8$. It follows that $D''$ is obtained from $D$ by the removal of at most $n/4$ vertices.
                
                Denote by $G$ the simple graph obtained from $D$ by ignoring directions and multiplicities, and let $G''$ be obtained similarly from $D''$. By assumption, $G$ has at least $100kn$ edges and maximum degree at most $200k$. It follows that $e(G'') \ge e(G) - \frac{n}{4} \cdot 200k \ge 50 k n$. Let $H''$ be an orientation of $G''$ which is a subdigraph of $D''$. Then $H''$ has average out-degree at least $50k$, implying that there is a vertex $u_j$ in $D''$ whose out-degree in $H''$ is at least $50k$. Let $W_j$ be a subset of the out-neighbourhood of $u_j$ in $H''$ of size $20k$. This shows that vertices $u_1, \ldots, u_a$ and sets $W_1, \ldots, W_a$ as above exist.
                A similar argument shows that there exist vertices $u_1', \ldots, u_a'$ and sets $W_1', \ldots, W_a'$, all in $D' \setminus (\{u_1, \ldots, u_a\} \cup W_{\le a})$, such that the sets $W_1, \ldots, W_a, W_1', \ldots, W_a'$ are pairwise disjoint sets of size $20k$ and $W_j$ is a set of in-neighbours of $u_j$. Take $\Up_i = \{u_1, \ldots, u_a\}$, $\Um_i = \{u_1', \ldots, u_a'\}$, $W(u_i) = W_i$ and $W(u_i') = W_i'$. 
            \end{proof}
            
            Let $\Up_1, \ldots, \Up_b, \Um_1, \ldots, \Um_b$ and $W(u)$, for $u \in U$, where $U = \bigcup_{i \in [b]}(\Up_i \cup \Um_i)$, be as in the claim above. 
            Note that $|U| = 2ab \le 4k$. 
            For $u \in U$, let $W'(u)$ be a subset of $W(u) \setminus U$ of size $10k$.
            
            For $u \in \Up$, let $F(u)$ be the set of edges in $D$ that touch $u$ but are not of the form $uv$ with $v \in W(u)$, and for $u \in \Um$, let $F(u)$ be the set of edges in $D$ that touch $u$ but are not of the form $vu$ with $v \in W(u)$. Let $F_0$ be the union of the sets $F(u)$ with $u \in U$. Then $|F_0| \le 200k \cdot 2ab \le 800k^2$.
            
            Let $M_1, \ldots, M_{ab}$ be a collection of perfect matchings in $\Up \times \Um$ that partition $\Up \times \Um$ and such that $M_i[\Up_j, \Um_{j'}]$ is either empty or a perfect matching for every $j, j' \in [b]$. We will find collections of paths $\PP_1, \ldots, \PP_{ab}$ as follows. Let $F_i$ be the union of $F_0$ with the edges that appear in one of the paths in $\PP_j$ with $j < i$. Take $\PP_i$ to be a maximal collection of pairwise edge-disjoint paths of length at most $(\log \ell)^5$ joining pairs of vertices in $M_i$, each joining a different pair. We claim that $|\PP_i| \ge k/2$ for $i \in [ab]$. Observe that, if true, this would complete the proof of the lemma.
            
            To see that $|\PP_i| \ge k/2$, fix $i \in [a]$ and suppose that $|\PP_i| \le k/2$. Let $M_i'$ be the submatching of $M_i$ consisting of pairs that are not joined by a path in $\PP_i$.
            Then $|M_i'| \ge k/2$, hence there exist $j, j' \in [b]$ such that $M_i[\Up_j, \Um_{j'}]$ is a perfect matching and $|M_i'[\Up_j, \Um_{j'}]| \ge a/2$. Denote $M' = M_i'[\Up_j, \Um_{j'}]$, let $X = V(M') \cap \Up_j$ and $Y = V(M') \cap \Um_{j'}$, and set $D' = D \setminus F_{i+1}$. For a vertex $u$ in $D'$ and integer $i$, let $\Bp_i(u)$ be the set of vertices $v$ for which there is a directed path of length at most $i$ from $u$ to $v$ in $D'$, and let $\Bm_i(u)$ be the set of vertices $v$ for which there is a directed path of length at most $i$ from $v$ to $u$ in $D'$.
            
            Set $r = (\log \ell)^4$.
            We claim that for every $(x, y) \in M'$ either $|\Bp_r(x)| \le k (\log \ell)^7$ or $|\Bm_r(y)| \le k (\log \ell)^7$. Indeed, suppose that $|\Bp_r(x)|, |\Bm_r(y)| \ge k(\log \ell)^7$ and let $X$ and $Y$ be subsets of $\Bp_r(x)$ and $\Bp_r(y)$, respectively, of size $k(\log \ell)^7$. Observe that $|F_{i+1}| \le d(D) \rho(|X|) |X|$. Indeed, this follows from the next two inequalities (using that the paths in $\PP_j$ have length at most $(\log n)^5$).
            \begin{align} \label{eqn:F}
                \begin{split}
                    & |F_{i+1}| \le |F_0| + (ab)^2 \cdot (\log \ell)^5 \le 800k^2 + 4k^2 (\log \ell)^5 \le k^2(\log \ell)^6, \\
                    & d(D) \rho(|X|) |X| 
                    \ge k \cdot \frac{1}{256(\log(4 (\log \ell)^7))^2} \cdot k (\log \ell)^7 \ge k^2 (\log \ell)^6.
                \end{split}
            \end{align}
            So, by \Cref{lem:short-paths-diam}, there is a directed path of length at most $1600(\log \ell)^3$ from $X$ to $Y$ in $D'$, showing that there is a path from $x$ to $y$ in $D'$ whose length is at most $1600(\log \ell)^3 + 2r \le (\log \ell)^5$, a contradiction to the maximality of $\PP_i$. 
            
            Hence, either $|\Bp_r(x)| \le k (\log \ell)^7$ for at least $a/4$ values of $x$ in $X$ or $|\Bm_r(y)| \le k (\log \ell)^7$ for at least $a/4$ values of $y$ in $Y$. Without loss of generality, we assume the former.
            Let $X_0$ be the set of vertices $x$ in $X$ for which $|\Bp_r(x)| \le k (\log \ell)^7$; so $|X_0| \ge a/4$. Define $X_i = \bigcup_{x \in X_0} \Bp_i(x)$. Since each $x \in X_0$ has no in-neighbours in $D \setminus F_0$, at most $ab$ of the paths in $\PP_1 \cup \ldots \cup \PP_i$ contain an edge touching $x$, implying that all but at most $ab \le 2k$ vertices of $W(x)$ are in $X_1$, for $x \in X_0$. Since the sets $W(x)$ are pairwise disjoint, we have $|X_1| \ge (a/4) \cdot (|W(x)| - 2k) \ge 2ak \ge \min\{k^2, \frac{2n}{(\log \ell)^8}\} \ge k (\log \ell)^{7}$ (using that $k \ge (\log \ell)^7$ and that $\ell$ is large). As in \eqref{eqn:F}, we have $|F_{i+1}| \le d(D) \rho(|X_1|)|X_1| \le d(D) \rho(|X_j|)|X_j|$ for $j \ge 1$. Thus, by expansion, if $|X_j| \le n/2$ then $|X_{j+1}| \ge (1 + \rho(|X_j|))|X_j| \ge (1 + \rho(n))|X_j|$. Suppose that $|X_r| \le n/2$. Then
            \begin{align*}
                |X_{r+1}| \ge (1 + \rho(n))^r |X_1|
                & = \left(1 + \frac{1}{256(\log (4n/k))^2}\right)^r \cdot |X_1| \\
                & \ge \left(1 + \frac{2}{(\log \ell)^3}\right)^r \cdot |X_1| 
                \ge \exp\left( \frac{r}{(\log \ell)^3} \right) k > k\ell = n,
            \end{align*}
            (recalling that $r = (\log \ell)^4)$), a contradiction. Hence $|X_r| \ge n/2 \ge a \cdot k(\log \ell)^7$ (using $ak \le \frac{n}{(\log \ell)^8}$ and that $\ell$ is large). By definition of $X_r$, it follows that $|\Bp_r(x)| \ge k (\log \ell)^7$ for some $x \in X_0$, a contradiction.
        \end{proof}

    \subsection{Immersions in Eulerian digraphs with high degree} \label{subsec:find-dense-immersion}
    
        \begin{theorem} \label{thm:find-dense-immersion}
            There exists a constant $\beta > 0$ such that for every large enough integer $k$ the following holds: every (simple) Eulerian digraph with minimum in- (and out-) degree at least $100 k$ immerses a simple digraph with at most $\beta k$ vertices and at least $k^2/2$ edges.
        \end{theorem} 
        
        \begin{proof}
            Let $\beta \ge 100$ be a sufficiently large constant so that \Cref{lem:immersion-small-n} holds when $n/k \ge \beta$ (and $k$ is large and $k \ge (\log(n/k))^7$). Let $k$ be a large enough integer, and let $D$ be an Eulerian digraph with minimum in-degree at least $100k$. By 
            \Cref{lem:get-regular-Eulerian}, $D$ immerses either $\dK_{50k, 50k}$ or a simple $100k$-regular Eulerian digraph (meaning that all in- and out-degrees are $100k$). If the former holds, we are done, so suppose that the latter holds, and let $D'$ be a simple Eulerian $100k$-regular digraph immersed by $D$. Now apply \Cref{lem:find-directed-expander} to find a multi-digraph $D''$ immersed by $D'$ with the following properties: $D''$ is a robust $k$-vertex-expander; and the simple graph obtained from $D''$ by ignoring directions and multiplicities has at least $100kn$ edges. Observe that by virtue of being an immersion of a $100k$-regular digraph, $D''$ has maximum in- and out-degree at most $100k$. We consider three cases: $n \le \beta k$; $n \ge \beta k$ and $k \ge (\log(n/k))^7$; and $n \ge 4k(100k)^{(\log \log n)^6}$. It is not hard to see that at least one of these three cases holds. Indeed, it suffices to show that if $k \le (\log (n/k))^7$ then $n \ge 4k(100k)^{(\log \log n)^6}$. The condition on $k$ implies $k \le (\log n)^7$, showing 
            \begin{equation*}
                \log\left(4k(100k)^{(\log \log n)^6}\right)
                \le 2 + \log k + (\log \log n)^6 \cdot \log (100k) 
                \le (\log \log n)^8
                \le \log n,
            \end{equation*}
            (using that $n$ is large, which follows from $k$ being large), as required.
            
            If $n \le \beta k$, then $D''$ is a graph on at most $\beta k$ vertices with at least $100kn \ge 5000 k^2$ edges, ignoring directions and multiplicities (using that $n \ge 50k$, which follows as the simple graph obtained by removing directions and multiplicities has average degree at least $50k$). If $n \ge \beta k$ and $k \ge (\log (n/k))^7$, by \Cref{lem:immersion-small-n}, $D''$ immerses a subgraph of $\dK_{2k,2k}$ with at least $k^2/2$ edges. Finally, if $n \ge 4k(100k)^{(\log \log n)^6}$, then by \Cref{lem:immersion-large-n}, $D''$ immerses $\dK_{k,k}$. Either way, $D''$ immerses a simple digraph on at most $\beta k$ vertices and with at least $k^2/2$ edges.
        \end{proof}
        
\section{Immersing a large complete digraph} \label{sec:immerse-clique}

    Our aim in this section is to prove the following theorem, which shows that a dense Eulerian multi-digraph immerses a large complete digraph.
    
    \begin{theorem} \label{thm:from-dense-to-complete}
        Let $D$ be an Eulerian multi-digraph on $n$ vertices whose underlying simple graph, obtained from $D$ by ignoring directions and multiplicities, has minimum degree at least $\alpha n$. Then $D$ immerses $\dK_s$, where $s = 10^{-9} \alpha^4 n$.
    \end{theorem}
    
    An important step in the proof of \Cref{thm:from-dense-to-complete} is to find short directed cycles in a graph with large minimum out-degree. To realise this step, we use the next lemma which finds short directed cycles in simple weighted digraphs with large minimum degree.
    
    A \emph{weighted digraph} is a digraph $D$ equipped with a weight function $\omega : V(D) \to \R^{\ge 0}$. Given a weighted digraph $D$ with weight function $\omega$ and a subset $U \subseteq V(D)$, denote $\omega(U) := \sum_{u \in U} \omega(u)$.

    \begin{lemma} \label{lem:short-cycle-weighted}
        Let $D$ be a weighted simple digraph (bi-directed edges are allowed) with weight function $\omega$, satisfying $\omega(\Np(u)) \ge \alpha \cdot \omega(V(D))$ for every vertex $u$. Then, there is a directed cycle of length at most $4\alpha^{-1}$.
    \end{lemma}
    
    \begin{proof}
        Let $U \subseteq V(D)$ be a minimal set satisfying that $\omega(\Np(u) \cap U) \ge \alpha \cdot \omega(U)$ for every $u \in U$. By possibly re-scaling $\omega$, we may assume that $\omega(U) = 1$.
        Write $D' = D[U]$. For a vertex $u \in U$ denote by $\Np_i(u)$ the set of vertices reachable from $u$ by a directed path of length at most $i$ in $D'$; define $\Nm_i(u)$ similarly. 
        
        Write $\ell = 2 \alpha^{-1}$. 
        We claim that $\omega(\Np_{\ell}(u)) \ge 2/3$ for every $u \in U$. Indeed, suppose that $\omega(\Np_{\ell}(u)) < 2/3$. Then some $i \in [\ell-1]$ satisfies $\omega(\Np_{i+1}(u) \setminus \Np_i(u)) \le \frac{2}{3\ell} \le \frac{\alpha}{3}$. This implies that for every vertex $v$ in $\Np_i(u)$ the following holds: $\omega(\Np(v) \cap \Np_i(u)) \ge \frac{2\alpha}{3} \ge \alpha \cdot \omega(\Np_i(u))$, contradicting the minimality of $U$.
        
        Next, we claim that there is a vertex $u$ for which $\omega(\Nm_{\ell}(u)) \ge 2/3$. Indeed, let $H$ be the weighted digraph on $U$ with weight function $\omega$, where $xy$ is an edge whenever there is a directed path of length at most $\ell$ in $D'$ from $x$ to $y$. Then $\omega(\Np_H(u)) \ge 2/3$ for every $u \in H$. We do a double counting as follows.
        \begin{equation*}
            \sum_{x \in U} \omega(x) \omega(\Nm_H(x))
            = \sum_{yx \in E(H)} \omega(x) \omega(y)
            = \sum_{y \in U} \omega(y) \omega(\Np_H(y))
            \ge \frac{2}{3} \cdot \sum_{y \in U}\omega(y)
            = 2/3.
        \end{equation*}
        It follows that there is a vertex $u$ with $\omega(\Nm_H(u)) \ge 2/3$; equivalently, $\omega(\Nm_{\ell}(u)) \ge 2/3$, as claimed.
        
        Let $u$ satisfy $\omega(\Nm_{\ell}(u)) \ge 2/3$. Since $\omega(\Np_{\ell}(u)) \ge 2/3$, there is a vertex $v$ such that $v \in \Np_{\ell}(u) \cap \Nm_{\ell}(u)$. Hence, there is a closed directed walk of length at most $2\ell$, implying that existence of a directed cycle of length at most $2\ell = 4\alpha^{-1}$.
    \end{proof}
    
    Next, we leverage \Cref{lem:short-cycle-weighted} to find directed cycles with few simple edges in digraphs with large minimum out-degree.
    
    \begin{lemma} \label{lem:short-cycle}
        Let $\alpha \in (0,1)$ and let $D$ be a multi-digraph (with no loops) on $n$ vertices with minimum out-degree at least $\alpha n$. Then there is a directed cycle with at most $4\alpha^{-1}$ simple edges. 
    \end{lemma}
    
    \begin{proof}
        Let $D'$ be the simple digraph on $V(D)$ where $xy$ is an edge whenever there are at least two directed edges in $D$ from $x$ to $y$. Let $X$ be the set of vertices with out-degree $0$ in $D'$. 
        
        \begin{claim}
            Either $D'$ contains a cycle or there is a partition $\{U(x) : x \in X\}$ of $V(D)$ such that $x \in U(x)$ and $x$ can be reached from each vertex in $U(x)$ in $D'$, for every $x \in X$. 
        \end{claim}
        
        \begin{proof}
            If $D'$ contains a directed cycle, we are done. We may therefore assume that this is not the case. 
            Write $X = \{x_1, \ldots, x_m\}$.
            Define subsets $U_1, \ldots, U_m \subseteq V(D')$ as follows. Suppose that $U_1, \ldots, U_{i-1}$ are defined. Let $D_i' = D' \setminus (U_1 \cup \ldots \cup U_{i-1})$ and let $U_i$ be the set of vertices $u$ for which there is a directed path from $u$ to $x_i$ in $D_i'$. Set $U(x_i) = U_i$. We claim that the collection $\{U(x): x \in X\}$ satisfies the requirements of the claim.
            
            First note that $x_i \in U_i$ for $i \in [m]$ (because there is no directed path in $D'$ between two distinct vertices in $X$, by choice of $X$). 
            Next, we show that $\{U_1, \ldots, U_m\}$ is a partition of $V(D')$. Indeed, clearly the sets $U_1, \ldots, U_m$ are pairwise disjoint. Now consider $u \in V(D')$. It is easy to see that there is a directed path from $u$ to $X$ (consider a maximal path from $u$ in $D'$, it must end in a vertex of out-degree $0$ since there are no directed cycles). Let $i$ be minimal such that there is a path $P$ from $u$ to $x_i$. Then none of the vertices in $P$ are in $U_1 \cup \ldots \cup U_{i-1}$ (by minimality of $i$). It follows that $x_i$ can be reached from $u$ in $D_i$, implying that $u \in U_i$. So $V(D) = U_1 \cup \ldots \cup U_m$.
        \end{proof}
        
        We assume that $D'$ is acyclic (otherwise $D$ has a cycle with no simple edges, as required).
        Note that $X$ is non-empty, as otherwise $D'$ contains a cycle. Let $\{U(x): x \in X\}$ be a partition of $V(D)$ as in the above claim.
        If there is an edge in $D$ from $x$ to $U(x)$ for some $x \in X$ then there is a directed cycle in $D$ with exactly one simple edge, as required. So suppose that there are no edges from $x$ to $U(x)$ for $x \in X$. 
       
        Let $H$ be an auxiliary weighted simple digraph on $X$, with weight function $\omega$ defined by $\omega(x) = |U(x)|$, and where $x y$ is an edge in $H$ whenever there is an edge in $D$ from $x$ to $U(y)$. By choice of $X$, every edge from $x$ to $U(y)$ in $D$ is a simple edge. By choice of $X$ and minimum out-degree assumption on $D$, we have $|\Np(x)| \ge \alpha n$ for $x \in X$. Thus
        \begin{equation*}
            \omega(\Np(x)) = \sum_{y \in X :\,\, U(y) \cap \Np(x) \neq \emptyset}|U(y)| \ge |\Np(x)| \ge \alpha n.
        \end{equation*}
        Since $\omega(X) = n$, it follows from \Cref{lem:short-cycle-weighted} that there is a cycle $C$ in $H$ of length at most $4\alpha^{-1}$. Write $C = (x_1 \ldots x_{\ell})$. For $i \in [\ell]$ let $y_i \in U(x_i)$ be such that $x_i y_{i+1}$ is an edge in $D$ (addition of indices is taken modulo $\ell$; such $y_i$ exists by definition of $H$). Let $P_i$ be a directed path in $D'$ from $y_i$ to $x_i$. Then $C' = (x_1 y_2 P_2 \ldots x_{\ell} y_1 P_1)$ is a closed walk in $D$ with $\ell$ simple edges. Since $\ell \le 4\alpha^{-1}$, this completes the proof. 
    \end{proof}
    
    Finally, we prove \Cref{thm:from-dense-to-complete}.
    
    \begin{proof}[Proof of \Cref{thm:from-dense-to-complete}]
    
        We first modify $D$ as follows. If there are three vertices $x, y, z$ such that both $xy$ and $yz$ are multiple edges, then remove a copy of $xy$ and $yz$ and add a copy of $xz$. Continue doing so until it is not longer possible and denote the resulting digraph by $D'$, and let $G'$ be the graph obtained from $D'$ by ignoring directions and multiplicities. Then $D'$ is an Eulerian multi-digraph which is immersed by $D$, and $G'$ has minimum degree at least $\alpha n$. Moreover, no vertex in $D'$ is incident to both multiple in-edges and multiple out-edges. Denote by $\Vm$ and $\Vp$ the vertices in $D'$ incident to multiple in-edges and multiple out-edges, respectively; so $\Vp$ and $\Vm$ are disjoint.
        
        \begin{claim}
            Let $\ell = 2^{-16} \alpha^4 n^2$. Then, there is a collection of $\ell$ pairwise edge-disjoint directed cycles in $D'$.
        \end{claim}
        
        \begin{proof}
            We define directed cycles $C_1, \ldots, C_{\ell}$ as follows. Suppose that $C_1, \ldots, C_{i-1}$ are chosen. Set $D_i = D' \setminus (C_1 \cup \ldots \cup C_{i-1})$ (here we take into account multiplicities), and let $C_i$ be a shortest directed cycle in $D_i$. 
            
            We claim that there is such a cycle $C_i$ and that $C_i$ has length at most $64 \alpha^{-1}$, for $i \in [\ell]$. 
            Suppose that this is the case for $j < i$, where $i \in [\ell]$. Let $X_i$ be the set of vertices that appear in at least $\alpha n / 4$ of the cycles $C_1, \ldots, C_{i-1}$. Then 
            \begin{equation*}
                |X_i| 
                \le \frac{\ell \cdot 64 \alpha^{-1}}{\alpha n / 4} 
                \le \frac{\alpha^2 n}{256}.
            \end{equation*} 
            Let $Y_i$ be the set of vertices with out-degree at least $\alpha n / 16$ in $X_i$. 
            Then $Y_i \subseteq \Vp$, because vertices outside of $\Vp$ send at most $|X_i|$ out-edges to $X_i$. 
            Note that the maximum in-degree in $D'$ is at most $n$, because $D'$ is Eulerian and every vertex has maximum in- or out-degree at most $n$. It follows that 
            \begin{equation*}
                |Y_i| \le \frac{|X_i|n}{\alpha n / 16} \le \frac{\alpha n}{16}.
            \end{equation*}
            Set $D_i' = D_i \setminus (X_i \cup Y_i)$.
            We claim that $D_i'$ has minimum out-degree at least $\alpha n / 8$. Indeed, let $u \in V(D_i')$. Observe that $u$ has out-degree at least $\alpha n/2$ in $D'$; this follows from $G'$ having minimum degree at least $\alpha n$.
            Thus, since $u$ is not in $X_i$, it has out-degree at least $\alpha n / 4$ in $D_i$. Moreover, since it is not in $Y_i$, it sends at most $\alpha n / 16$ out-edges to $X_i$. Additionally, $u$ sends at most $|Y_i|$ out-edges to $Y_i$, because $Y_i \subseteq \Vp$. It follows that $u$ has out-degree at least $\alpha n / 8$ in $D_i'$, as required.
            
            By \Cref{lem:short-cycle}, there is a directed cycle $C$ in $D_i'$ all but at most $32\alpha^{-1}$ of its edges are simple. By the structure of $D'$, there cannot be two consecutive multiple edges in $C$. It follows that $C$ has length at most $64\alpha^{-1}$. This implies that the cycle $C_i$ exists and has length at most $64 \alpha^{-1}$, as required.
        \end{proof}
        
        Let $\CC$ be a collection of at least $\ell$ pairwise edge-disjoint directed cycles in $D'$. For each $C \in \CC$, let $e(C)$ be an edge of $C$ which is simple in $D'$ (the structure of $D'$ implies that such an edge exists). 
        We define a graph $H$ on $V(D')$ as follows. For each $C \in \CC$, add the edge $e(C)$ to $H$, ignoring its direction. Then $e(H) \ge \ell = 2^{-16} \alpha^4 n^2$, so $H$ has a subgraph with minimum degree at least $2^{-16} \alpha^4 n$. 
        By \Cref{thm:devos-et-al}, the graph $H$ immerses $K_s$, where $s = 10^{-9} \alpha^4 n$. This means that there is a set $X$ of $s$ vertices and a path $P_{xy}$ from $x$ to $y$ for every two vertices $x, y \in X$, such that these paths are pairwise edge-disjoint. We show that $D'$ immerses $\dK_s$.
        
        Fix two vertices $x, y \in X$. Write $P_{xy} = (x_0, \ldots, x_r)$. For each $i \in [r]$, let $C_i$ be the cycle in $\CC$ for which $e(C_i) = x_{i-1} x_i$, let $Q_i$ be the subpath of $C_i$ from $x_{i-1}$ to $x_i$ and let $Q_i'$ be the subpath of $C_i$ from $x_i$ to $x_{i-1}$ (one of $Q_i$ and $Q_i'$ is an edge). Define $D_{xy} = (Q_1 \ldots Q_r)$ and $D_{yx} = (Q_r' \ldots Q_1')$. Then $D_{xy}$ is a directed path from $x$ to $y$ and $D_{yx}$ is a directed path from $y$ to $x$. It is easy to see that the paths $D_{xy}$, with $x,y \in X$ and $x \neq y$, are pairwise edge-disjoint. The union of these paths yields an immersion of $\dK_s$ in $D'$, as required.
    \end{proof}
    \section{Proof of Theorem~\ref{thm:main}}\label{sec: proofmain}
    
     \begin{proof}
        We may assume $t$ is large by taking $\alpha$ to be a sufficiently large constant and using that an Eulerian simple digraph with minimum out-degree $t^2$ contains an immersion of $\dK_t$, as proved in \cite{devoss2010immersing}.
        
        Let $\beta > 0$ be a constant as in \Cref{thm:find-dense-immersion}, and write $k = 10^{11} \beta^8 t$. We will show that every simple Eulerian digraph with minimum in-degree at least $100k$ immerses $\dK_t$, showing that the statement holds with $\alpha = 10^{13} \beta^8$. 
        
        Let $D$ be a simple Eulerian digraph with minimum in-degree at least $100k$.
        By \Cref{thm:find-dense-immersion}, $D$ immerses a simple digraph $D'$ on at most $\beta k$ vertices and at least $k^2/2$ edges. Let $G'$ be the graph obtained from $D'$ by ignoring directions, let $G''$ be a subgraph of $G'$ with minimum degree at least $d(G')/2$, and let $D''$ be an orientation of $G''$ which is a subgraph of $D'$. Write $n = |G''|$ and $\alpha = 1/2\beta^2$. Note that $d(G') = 2e(D')/|D'| \ge k/\beta$, showing $\delta(G'') \ge k/2\beta \ge n / 2\beta^2 = \alpha n$. Applying Lemma~\ref{lem:complete-to-eulerian}, with $D''$ playing the role of $D'$, we obtain an Eulerian multi-digraph $D'''$ on the same vertex set as $D''$ which contains $D''$ as a subdigraph and is immersed by $D$. 
        By \Cref{thm:from-dense-to-complete}, $D'''$ immerses $\dK_s$, where $s = 10^{-9} \alpha^4 n \ge \frac{k}{10^9 2^4 \beta^8} \ge t$, as claimed. 
    \end{proof}
    
\section{Concluding remarks} \label{sec:conclusion}

    As stated in the introduction, Lochet proved that for every positive integer $k$, there exists $f(k)$ such that any digraph with minimum out-degree at least $f(k)$ contains a immersion of a transitive tournament on $k$ vertices. This is essentially best possible since, as we already pointed out, there are digraph with arbitrarily large minimum-out degree which do not contain an immersion of a $\dK_3$ (see \cite{mader1985degree,devos2012}).
    
    Lochet's proof allowed him to show $f(k)\leq O(k^3)$. We suspect that a linear bound would suffice. 
    
    \begin{conjecture}
        There exists an absolute constant $C>0$ such that for any positive integer $k$ the following holds. Let $D$ be a digraph with $\delta^{+}(D)\geq Ck$. Then $D$ immerses a transitive tournament on $k$ vertices. 
    \end{conjecture}
    
    To conclude, we reiterate Mader's question about subdivisions of transitive tournaments in digraphs with large minimum out-degree.
    
    \begin{question}[Mader \cite{mader1996topological}]
        Is there a function $f$ such that, for every integer $k \ge 1$, every digraph with minimum out-degree at least $f(k)$ contains a subdivision of a transitive tournament on $k$ vertices?
    \end{question}
    
    \section{Acknowledgements}
    We would like to thank Paul Wollan for an insightful conversation on the topic of immersions. 

    \bibliography{immersion.bib}
    \bibliographystyle{amsplain}
    
\appendix
    
\section{Expander lemmas} \label{appendix}
    
    \subsection{Proof of \Cref{lem:find-expander}}
    \begin{proof}
        Define $\gamma(x) = \min\{1, \frac{1}{\log(2x/t)}\}$ and $\rho(x) := \rho_t(x)$ for $x > 0$.
        For a graph $H$ define $\phi(H) := d(H)(1 + \gamma(|H|))$.
        
        Let $H$ be a subgraph of $G$ satisfying $\phi(H) = \max\{\phi(G') : G' \subseteq G\}$. We will show that $H$ satisfies the requirements of the lemma.
        
        Write $n := |H|$ and $d := d(H)$.
        First, by choice of $H$, we have $\phi(H) \ge \phi(G)$, implying that $d = d(H) \ge \frac{d(G)(1 + \gamma(|G|)}{1 + \gamma(|H|)} \ge \frac{d(G)}{2}$, using that $\gamma(x) \in [0,1]$ for $x > 0$.
        For a subset $S \subseteq V(H)$, write $d(S) := d(H[S])$.
        \begin{claim} \label{claim:d-min}
            Let $S \subseteq V(H)$. Then $d(S) \le d$. Moreover, if $t \le |S| \le \frac{2n}{3}$ then $d(S) \le d \cdot (1 - 64\rho(|S|))$.
        \end{claim}
        \begin{proof}
            By choice of $H$ we have $d \cdot (1 + \gamma(n)) = d(H)(1 + \gamma(|H|)) \ge d(S)(1 + \gamma(|S|)) \ge d(S)(1 + \gamma(n))$, where the second inequality follows from $\gamma$ being non-increasing. It follows that $d(S) \le d$, as claimed.
            
            Now suppose that $t \le |S| \le \frac{2n}{3}$, and write $s := |S|$. Again by choice of $H$, we have $d(S) \le d \cdot \frac{1 + \gamma(n)}{1 + \gamma(s)}$. Using that $\gamma$ is non-increasing, $\gamma(x) \le 1$ for every $x$, and $t \le s \le \frac{2n}{3}$, we obtain the following chain of inequalities.
            \begin{align*}
                \frac{1 + \gamma(n)}{1 + \gamma(s)}
                & \le \frac{1 + \gamma(3s/2)}{1 + \gamma(s)}
                = 1 - \frac{\gamma(s) - \gamma(3s/2)}{1 + \gamma(s)}
                \le 1 - \frac{\gamma(s) - \gamma(3s/2)}{2}
                = 1 - \frac{\frac{1}{\log(2s/t)} - \frac{1}{\log(3s/t)}}{2} \\
                & = 1 - \frac{\log(3s/t) - \log(2s/t)}{2\log(2s/t)\log(3s/t)}
                \le 1 - \frac{1}{4(\log(4s/t))^2}
                = 1 - 64\rho(s),
            \end{align*}
            implying that $d(S) \le d \cdot (1 - 64\rho(s))$, as claimed.
        \end{proof}
        
        \begin{claim} \label{claim:min-deg}
            Let $S \subset V(H)$. Then $e(S) + e(S, S^c) \ge \frac{d|S|}{2}$. In particular, $\delta(H) \ge \frac{d}{2} \ge \frac{d(G)}{4}$.
        \end{claim}
        \begin{proof}
            Suppose not, then $e(S^c) > \frac{d|S^c|}{2}$, contradicting the previous claim.
        \end{proof}
        Let $S \subseteq V(H)$ satisfy $t \le |S| \le n/2$. The last two claims imply that 
        \begin{equation*} 
            e(S, S^c) 
            \ge \frac{d|S|}{2} - e(S)
            = \frac{(d - d(S))\cdot |S|}{2}
            \ge 32d \cdot \rho(|S|)|S|,
        \end{equation*}
        so the required edge-expansion property holds (note that it holds vacuously when $|S| < t$). 
        
        To prove the vertex-expansion property, let $F \subseteq H$ satisfy $e(F) \le d \cdot \rho(|S|)|S|$ and write $T = S \cup N_{G \setminus F} (S)$. If $|T| \ge \frac{4}{3}|S|$ then $|N_{G \setminus F}(S)| \ge \frac{1}{3}|S| \ge 2\rho(|S|)|S|$ (using $\rho(x) \le 1/256$ for all $x$), as required. So we may assume $|T| \le \frac{4}{3}|S| \le \frac{2n}{3}$. By \Cref{claim:d-min}, $d(T) \le d \cdot (1 - 64\rho(|T|))$, implying that 
        \begin{equation*}
            e(T) \le \frac{d|T|(1 - 64\rho(|T|))}{2}.
        \end{equation*}
        By \Cref{claim:min-deg},
        \begin{equation*}
            e(T) 
            \ge e(S) + e(S, S^c) - e(F)
            \ge \frac{d|S|}{2} - e(F) 
            \ge \frac{d|S|(1 - 2\rho(|S|)}{2}.
        \end{equation*}
        Putting the two inequalities we get $|S|(1 - 2\rho(|S|)) \le |T|(1 - 64\rho(|T|))$, implying the following, where $s := |S|$,
        \begin{equation*}
            \frac{|T|}{|S|} 
            \ge \frac{1 - 2\rho(|S|)}{1 - 64\rho(|T|)}
            \ge \frac{1 - 2\rho(s)}{1 - 64\rho(2s)}
            = 1 + \frac{64\rho(2s) - 2\rho(s)}{1 - 64\rho(2s)}
            \ge 1 + \frac{14\rho(s)}{1 - 64\rho(2s)}
            \ge 1 + 14\rho(s),
        \end{equation*}
        where we used the inequality $4\rho(2s) \ge \rho(s)$. To see this, since $s \ge t$, it suffices to observe that $\frac{1}{2}\log(8s/t) = \frac{1}{2}(\log(4s/t) + 1) \le \log(4s/t)$.
        It follows that $|N_{G \setminus F}(S)| = |T| - |S| \ge 14\rho(|S|)|S| > \rho(|S|)|S|$, as claimed.
    \end{proof}
    
    \subsection{Proof of \Cref{lem:short-paths-diam}}
    \begin{proof}
        Write $D' = D \setminus F$.
        Let $X_i$ be the set of vertices $x$ for which there is a directed path (in $D'$) of length at most $i$ that starts at $X$ and ends in $x$, and let $Y_i$ be the set of vertices $y$ for which there is a directed path of length at most $i$ that starts at $y$ and ends in $Y$.  
        
        By expansion, if $|X_i| \le n/2$ then 
        \begin{align*}
            |X_{i+1}| 
            \ge (1 + \rho(|X_i))|X_i|
            & = \left(1 + \frac{1}{256(\log(4|X_i|/t))^2}\right)|X_i| \\
            & \ge \left(1 + \frac{1}{256(\log(4n/t))^2}\right)|X_i|
            \ge \left(1 + \frac{1}{400(\log (n/t))^2}\right)|X_i|,
        \end{align*}
        using that $\log(4n/t) = 2 + \log(n/t) \le \frac{5}{4}\log(n/t)$ (where the inequality follows as $n/t \ge 2^8$).
        Hence, if $i$ satisfies $|X_{i-1}| \le n/2$, then
        \begin{equation*}
            |X_{i}| \ge \left(1 + \frac{1}{400(\log (n/t))^2}\right)^i |X|
            \ge \exp\left(\frac{i}{800(\log (n/t))^2}\right)|X|.
        \end{equation*}
        Write $\ell = 800(\log (n/t))^3$. We claim that $|X_{\ell}| > n/2$. Indeed, otherwise $|X_{\ell}| \ge \exp(\log (n/t))|X| > n$, a contradiction.
        An analogous argument shows that $|Y_{\ell}| > n/2$. It follows that there is a directed path of length at most $2\ell = 1600(\log (n/t))^3$ from $X$ to $Y$.
    \end{proof}

\end{document}